\documentclass[11pt]{article}

\usepackage{amssymb,amsmath}
\usepackage{amsthm}
\usepackage{hyperref}
\usepackage{mathrsfs}
\usepackage{textcomp}
\usepackage{verbatim}
\hypersetup{
    colorlinks,%
    citecolor=black,%
    filecolor=black,%
    linkcolor=black,%
    urlcolor=black
}

\usepackage{sectsty}

\makeatletter

\newdimen\bibspace
\renewenvironment{thebibliography}[1]{%
 \section*{\refname 
       \@mkboth{\MakeUppercase\refname}{\MakeUppercase\refname}}%
     \list{\@biblabel{\@arabic\c@enumiv}}%
          {\settowidth\labelwidth{\@biblabel{#1}}%
           \leftmargin\labelwidth
           \advance\leftmargin\labelsep
           \itemsep\bibspace
           \parsep\z@skip     %
           \@openbib@code
           \usecounter{enumiv}%
           \let\p@enumiv\@empty
           \renewcommand\theenumiv{\@arabic\c@enumiv}}%
     \sloppy\clubpenalty4000\widowpenalty4000%
     \sfcode`\.\@m}
    {\def\@noitemerr
      {\@latex@warning{Empty `thebibliography' environment}}%
     \endlist}

\makeatother

\makeatletter

\newtheorem{thm}{Theorem}[section]
\newtheorem{lem}{Lemma}[section]
\newtheorem{prop}{Proposition}[section]

\newtheorem{cor}{Corollary}[section]

\def\XXint#1#2#3{{\setbox0=\hbox{$#1{#2#3}{\int}$}
  \vcenter{\hbox{$#2#3$}}\kern-.5\wd0}}

\newcommand{\al}{\alpha}                \newcommand{\lda}{\lambda}
\newcommand{\om}{\Omega}                \newcommand{\pa}{\partial}
\newcommand{\va}{\varepsilon}           \newcommand{\ud}{\mathrm{d}}
\newcommand{\be}{\begin{equation}}      \newcommand{\ee}{\end{equation}}
\newcommand{\w}{\omega}                 \newcommand{\ep}{\epsilon}
              
\newcommand{\R}{\mathbb{R}}

\title{\textbf{A sharp Sobolev trace inequality involving the mean curvature on Riemannian manifolds}}

\author{\medskip Tianling Jin \  and \
Jingang Xiong}

\begin{document}

\maketitle

\begin{abstract}
In this paper, we examine the boundary $L^2$ term of the sharp Sobolev trace inequality $\|u\|_{L^{q}(\pa M)}^2\leq S \|\nabla_g u\|_{L^2(M)}^2 +A(M,g)\|u\|^2_{L^2(\pa M)}$ on Riemannian manifolds $(M,g)$ with boundaries $\pa M$, where $q=\frac{2(n-1)}{n-2}$, $S$ is the best constant and $A(M,g)$ is some positive constant depending only on $M$ and $g$. We obtain a sharp trace inequality involving the mean curvature in a remainder term, which would fail in general once the mean curvature is replaced by any smaller function.
\end{abstract}

\section{Introduction} 

Recall that the sharp Sobolev trace inequality in the upper half space asserts
\be\label{eq:eucl1}
\left(\int_{\pa \R^n_+} |u|^{\frac{2(n-1)}{n-2}}\,\ud x'\right)^{\frac{n-2}{n-1}}\leq S\int_{\R^n_+}|\nabla u |^2\,\ud x\quad \forall ~u\in C_c^1(\R^n_+\cup \pa\R^n_+ ), 
\ee
where $S=
\frac{2}{n-2}\w_n^{-1/(n-1)}$, $x=(x', x_n)$, 
and $\w_n$ is the volume of the unit sphere in $\R^n$, $n\ge 3$. The best constant and the extremal functions of \eqref{eq:eucl1} were found by Escobar \cite{E88} and Beckner \cite{Be} independently. Indeed, as pointed out in \cite{Be}, \eqref{eq:eucl1} follows from the sharp fractional Sobolev inequality, because
\be\label{eq:fineq}
\left(\int_{\R^{n-1}} v^{\frac{2(n-1)}{n-2}}\,\ud x'\right)^{\frac{n-2}{n-1}}\leq S \|v\|^2_{\dot{H}^{1/2}(\R^{n-1})}
\ee
and 
\[
 \|v\|^2_{\dot{H}^{1/2}(\R^{n-1})} = \int_{\R^n_+}|\nabla u |^2\,\ud x, 
\]
where $u$ is the harmonic extension of $v$ in the upper half space, and the equality of \eqref{eq:fineq} holds if and only if $v(x')$ takes the form  
\[
v(x')=c \left(\frac{\lda}{1+\lda^2|x'-x'_0|}\right)^{\frac{n-2}{2}} 
\]
for some $c\in \R$, $\lda>0$ and $x_0'\in \R^{n-1}$, which is proved by Lieb \cite{Lie83}.

For any bounded smooth domain $\om\subset\R^n$,  the sharp trace inequality
\be \label{eq:al}
\|u\|^2_{L^{\frac{2(n-1)}{n-2}}(\pa \om)}\leq S \|\nabla u\|^2_{L^2(\om)}+A(\om) \|u\|^2_{L^2(\pa \om)}\quad \forall~ u\in H^1(\om),
\ee 
where $A(\om)>0$ depends only on $\om$, was established by Adimurthi-Yadava \cite{AY} for $n\ge 5$ and by Li-Zhu \cite{LZ97} on Riemannian manifolds in all dimensions. The following sharp Sobolev inequality  
\be\label{eq:bl}
\|u\|^2_{L^{\frac{2n}{n-2}}(\om)}\leq 2^{\frac{2}{n}}S_n \|\nabla u\|^2_{L^2(\om)}+A(\om) \|u\|^2_{L^2(\pa \om)}\quad \forall~ u\in H^1(\om),
\ee
where $S_n$ is the best Sobolev constant, was initially proved by Brezis-Lieb \cite{BL} for $\om$ being the unit ball, and later by Adimurthi-Yadava \cite{AY} for $n\geq 5$ and by Li-Zhu \cite{LZ98} on Riemannian manifolds in all dimensions.  We also obtained a similar sharp weighted trace inequalities in \cite{JX_b}, and some remainder terms in sharp fractional Sobolev inequalities in \cite{JX_a}.

In this and a forthcoming paper, we would like to examine the boundary $L^2$ term in \eqref{eq:al} and \eqref{eq:bl} further. The main result of the present paper is as follows.

\begin{thm} Let $\om$ be a bounded smooth domain in $\R^n$ with $n\geq 5$. There exists a positive constant $A(\om)$ depending only on $\om$ such that for all $u\in H^1(\Omega)$,
\be\label{eq:main tr}
\left(\int_{\pa \om}|u|^q\,\ud s\right)^{2/q}\leq S\left(\int_{\om}|\nabla u|^2\,\ud x+ \frac{n-2}{2}\int_{\pa \om}h u^2\,\ud s\right)+A(\om)\|u\|^2_{L^r(\pa \om)}, 
\ee
where $q=\frac{2(n-1)}{n-2}$, $r=q'=\frac{2(n-1)}{n}$ and $h(x)$ is the mean curvature of $\pa \om$ at $x$.  
\end{thm}

The inequality \eqref{eq:main tr} also holds on Riemannian manifolds with boundaries. Namely, 

\begin{thm}\label{thm:Tr}  Let $(M,g)$ be a smooth compact $n$-dimensional Riemannian manifold with smooth boundary $\pa M$ for $n\ge 5$.
Then there exists a positive constant $A(M,g)$ depending only on $M, g$ such that for all $u\in H^1(M)$,
\be\label{main tr}
\left(\int_{\pa M}|u|^q\,\ud s_g\right)^{2/q}\leq S\left(\int_{M}|\nabla_g u|^2\,\ud v_g+ \frac{n-2}{2}\int_{\pa M}h_g u^2\,\ud s_g\right)+A\|u\|^2_{L^r(\pa M)}, 
\ee  where $h_g$ is the mean curvature of $\pa M$ with respect to the metric $g$, $\ud v_g$ is the volume form of $(M,g)$ and $\ud s_g$ is the induced volume form on $\pa M$.
\end{thm}

The inequality \eqref{main tr} is sharp in the following sense. It would fail if $S$ is replaced by a smaller constant. In general, $h$ can not be replaced by any function which is smaller than $h$ at some point on $\pa M$, and $r$ can not be replaced by any smaller number.

The best constant $S$ has been demonstrated to be crucial in study of the Yamabe problem on manifolds with boundaries, see, e.g., Escobar \cite{E92a}.  The effect of mean curvatures in sharp Hardy-Sobolev inequality with a singularity on the boundary has been studied by Ghoussoub-Robert \cite{GR}.  In \cite{JX_b}, we proved a sharp weighted Sobolev trace inequality on Riemannian manifolds which would fail if the mean curvature is positive somewhere. The effect of scalar curvatures for sharp Sobolev inequalities on compact manifolds has been studied by Li-Ricciardi \cite{LR}, Hebey \cite{H_b} and references therein.

The sharp Sobolev inequality \eqref{main tr} is in the same spirit of a conjecture posed by Aubin \cite{Aubin76}, which has been confirmed through the work of Hebey-Vaugon \cite{HV}, Aubin-Li \cite{AL} and Druet \cite{Druet99}, \cite{Druet02}, see also  Druet-Hebey  \cite{DH} and Hebey \cite{H}.  The procedure to prove those types of inequalities is by contradiction. A key point is to derive the asymptotical behavior of extremal functions near their energy concentration points. However, in order to establish \eqref{main tr} it requires more precise estimates, in particular the error estimates between the extremal functions and some properly chosen bubbles. 

\bigskip

\noindent\textbf{Acknowledgements:}  Both authors would like to thank Professor YanYan Li
for his interest in the work and constant encouragement. Jingang Xiong was supported in part by the First Class Postdoctoral Science Foundation of China (No. 2012M520002).

\section{Preliminaries}

We prove inequality \eqref{main tr} by contradiction. Suppose that for every large $\al >1$, there exists $\tilde{u}\in H^1(M)$ such that
\[
 \left(\int_{\pa M}|\tilde{u}|^q\,\ud s_g\right)^{2/q}> S\left(\int_{M}|\nabla_g \tilde{u}|^2\,\ud v_g+ \frac{n-2}{2}\int_{\pa M}h_g \tilde{u}^2\right)+\al\|\tilde{u}\|^2_{L^r(\pa M,g)}.
\]
Define
\[
 I_\al(u)=\frac{\int_{M}|\nabla_g u|^2\,\ud v_g+ \frac{n-2}{2}\int_{\pa M}h_g u^2+ \al\|u\|^2_{L^r(\pa M,g)}}{\|u\|^2_{L^q(\pa M)}},
\quad \forall \ u\in H^1(M)\setminus H^1_0(M).
\]
It follows from the contradiction hypothesis that for all large $\al$,
\be\label{cont}
\ell_\al:=\inf_{H^1(M)\setminus H^1_0(M)}I_\al <\frac{1}{S}.
\ee
Even though the functional $I_\al$ involves critical Sobolev exponent, the above strict inequality implies the existence of a minimizer, namely,

\begin{prop}\label{existence of a minimizer} For every $\al >0$, $\ell_\al$ is achieved by a nonnegative $u_\al\in H^1(M)\setminus H^1_0(M)$ with
\[
 \int_{\pa M}u^q_\al\,\ud s_g=1.
\]
Consequently, $u_\al \in C^\infty(M)\cap C^{1,r-1}(\overline M)$ satisfies the Euler-Lagrange equation
\be\label{el equ}
\begin{cases}
- \Delta_g u_\al =0&\quad \mbox{on } M,\\
\frac{\pa_gu_\al}{\pa \nu}+\frac{n-2}{2}h_gu_\al+\al \|u_\al\|^{2-r}_{L^r(\pa M)}u_\al^{r-1}=\ell_\al u_\al^{q-1}&\quad  \mbox{on }\pa M,
\end{cases}
\ee
where $\pa_g/\pa \nu$ denotes the differentiation in the direction of the unit outer normal of $\pa M$ with respect to $g$.
\end{prop}
\begin{proof}
The existence of minimizer follows from the standard subcritical exponent approximating method and Moser's iteration argument, see, e.g., Proposition 2.1 of \cite{E92a}.
A calculus of variations argument shows that $u_\al$ is a weak solution of Euler-Lagrange equation \eqref{el equ}. Finally, that $u_\al \in C^\infty(M)\cap C^{1,r-1}(\overline M)$ follows from the regularity theory in \cite{Ch}.
\end{proof}

\begin{lem}\label{lem:su ineq1}
 For every $\va>0$, there exists a constant $B(\va)>0$ depending on $\va$ such that
\[
 \left(\int_{\pa M}|u|^q\,\ud s_g\right)^{2/q}\leq (S+\va)\int_{M}|\nabla_g u|^2\,\ud v_g+B(\va)\|u\|^2_{L^r(\pa M,g)},
\]
for all $u\in H^1(M)$.
\end{lem}
\begin{proof}
 By the compactness, we have that for every $\va>0$ there exists a positive constant $\tilde B(\va)$ such that
\be\label{interp}
 \int_{\pa M}u^2\,\ud s_g\leq \va \int_{M}|\nabla_g u|^2\,\ud v_g +\tilde B(\va) \|u\|^2_{L^r(\pa M,g)}.
\ee
Hence, the proposition follows from the sharp trace inequality in Theorem 0.1 in \cite{LZ97}.
\end{proof}

Note that
\[
 \begin{split}
  I_\al(u_\al)&=\int_{M}|\nabla_g u_\al|^2\,\ud v_g+ \frac{n-2}{2}\int_{\pa M}h_g u^2_\al+ \al\|u_\al\|^2_{L^r(\pa M,g)}\\&
\geq (1-\va \max_{\pa M}|h_g|) \int_{M}|\nabla_g u_\al|^2\,\ud v_g+(\al -B(\va)) \|u_\al\|^2_{L^r(\pa M,g)}.
 \end{split}
\]
This implies
\[
 \int_{M}|\nabla_g u_\al|^2\,\ud v_g \leq \frac{2}{S}
\]
and
\[
 \|u_\al\|_{L^r(\pa M,g)}\to 0\quad \mbox{as }\al \to \infty.
\]
It follows that $u_\al \rightharpoonup \bar u$ in $H^1(M)$ for some $\bar u\in H^1_0(M)$. In fact, $\bar u=0, a.e.$ on $M$.
This is because $u_\al$ is harmonic on $M$ and thus it satisfies (see Proposition 2.1 in \cite{HWY})
\[
 \|u_\al\|_{L^\frac{nr}{n-r}(M,g)}\leq C(M,g, n)\|u_\al\|_{L^r(\pa M,g)},
\]
and $u_\al\to \bar u$ in $L^{^\frac{nr}{n-r}(M)}$.

We claim that, as $\al \to \infty$,
\be\label{l converges}
\ell_\al \to \frac{1}{S}
\ee
and
\be\label{lower term converges}
 \al\|u_\al\|^2_{L^r(\pa M,g)}\to 0.
\ee
Indeed, by Lemma \ref{lem:su ineq1} and \eqref{interp} for every $\va>0$
\[
 \begin{split}
  1&\leq (S+\frac{\va}{2})  \int_{M}|\nabla_g u_\al|^2\,\ud v_g+B(\va)\|u_\al\|^2_{L^r(\pa M,g)}\\ &
\leq (S+\va)\ell_\al +(2 B(\va)-\al S) \|u_\al\|^2_{L^r(\pa M,g)}.
 \end{split}
\]
Thus
\[
 \frac{1}{S+\va}\leq \ell_\al <\frac{1}{S},\quad \frac{1}{2}\al S \|u_\al\|^2_{L^r(\pa M,g)}\leq (S+\va)\frac{1}{S}-1=\frac{\va}{S},
\]
if $\al S>4B(\va)$. Hence,  the claim follows.

By the maximum principle, there exists a point $x_\al\in \pa M$ such that
\[
 u_\al(x_\al)=\max_{\overline M} u_\al>0.
\]
In view of \eqref{lower term converges} and
\[
 1=\int_{\pa M}u^q_\al\leq  u_\al(x_\al)^{q-r} \int_{\pa M}u^r_\al,
\]
we have $ u_\al(x_\al)\to \infty$. Set $\mu_\al:= u_\al(x_\al)^{-2/(n-2)}$. Then $\mu_\al \to 0$ as $\al \to \infty$.

Let $x=(x_1,\cdots, x_{n-1}, x_n)=(x',x_{n})$ be \emph{Fermi coordinate} (see, e.g., \cite{E92a}) at $x_\al$, where $(x_1, \cdots, x_{n-1})$
are normal coordinates on $\pa M$ at $x_\al$ and $\gamma(x_{n})$ is the geodesic leaving from $(x_1,\cdot, x_{n-1})$ in the orthogonal direction to $\pa M$ and parameterized by arc length.
In this coordinate system,
\[
\sum_{1\leq i,j\leq n}g_{ij}(x)\ud x_i\ud x_j =\ud x_{n}^2+\sum_{1\leq i,j\leq n-1}g_{ij}(x)\ud x_i\ud x_j.
\]
Moreover, $g^{ij}$ has the following Taylor expansion near $\pa M$:
\begin{lem}[Lemma 3.2 in \cite{E92a}]\label{taylor expansion of fermi}
For $\{x_k\}_{k=1,\cdots, n}$ are small,
\begin{subequations}
\begin{align}\label{eq:taylor expansion of fermi}
g^{ij}(x)&=\delta^{ij}+2h^{ij}(x',0)x_{n}+O(|x|^2),\\
 g^{ij}\Gamma_{ij}^k&=O(|x|).
\end{align}
\end{subequations}
where $i,j=1,\cdots,n-1$ and $h_{ij}$ is the second fundamental form of $\pa M$.
\end{lem}

 To proceed, we introduce some notations. 
 \begin{itemize}
 \item For a domain $D\subset \R^{n}$ with boundary $\pa D$, we denote $\pa' D$ as the interior of $\overline D\cap \pa \R^{n}_+$ in $\R^{n-1}=\partial\R^{n}_+$ and $\pa''D=\pa D \setminus \pa' D$. 
 
\item For $\bar x\in \R^{n}$, $B_{r}(\bar x):=\{x\in \R^{n}: |x-\bar x|<r\}$ and $B^+_{r}(\bar x):=B_{r}(\bar x)\cap \R^{n}_+$, where $|\cdot|$ is the Euclidean distance. 
We will not keep writing the center $\bar x$ if $\bar x=0$. 
\end{itemize}
For suitably small $\delta_0>0$ (independent of $\al$), we define $v_\al$ in a neighborhood of
$x_\al=0$ by
\[
 v_\al(y)=\mu_\al^{(n-2)/2}u_\al(\mu_\al y),\quad x\in B_{\delta_0/\mu_\al}^+,
\]
in the Fermi coordinate mentioned above. Then $v_\al$ satisfies
\be\label{eq 0}
\begin{cases}
\Delta_{g_\al}v_\al=0,&\quad \mbox{in }B_{\delta_0/\mu_\al}^+\\
\frac{\pa_{g_\al} v_{\al}}{\pa \nu}+\frac{n-2}{2}h_\al v_\al+\ep_\al v_\al^{r-1}=\ell_\al v_\al^{q-1},& \quad \mbox{on }\pa' B_{\delta_0/\mu_\al}^+\\
v_\al(0)=1,\quad  0\leq v_\al \leq 1,
\end{cases}
\ee
where $g_\al(x)=g_{ij}(\mu_\al x)\ud x_i\ud x_j$, $h_\al$ is the mean curvature of $\pa' B_{\delta_0/\mu_\al}$ with respect to the metric $g_\al$ and
\be\label{eq:epsilon alpha}
 \ep_\al:=\al \mu_\al^{n-1-\frac{n-2}{2}r}\|u_\al\|^{2-r}_{L^r(\pa M)}.
\ee
Note that
\be\label{etimate for ep}
 \ep_\al=\frac{\al \|u_\al\|^{2}_{L^r(\pa M)}}{(\max_{\overline{M}}u_\al)^{q-r}\int_{\pa M}u_\al^r\,\ud s_g}\leq \al \|u_\al\|^{2}_{L^r(\pa M)}\to 0\quad\mbox{as }\al \to \infty,
\ee
because
$1=\int_{\pa M}u_\al^q\,\ud s_g\leq  (\max_{\overline{M}}u_\al)^{q-r}\int_{\pa M}u_\al^r\,\ud s_g$. By the standard elliptic equations theory,
for all $R>1$,
\be\label{holder estimate 1}
\|v_\al\|_{C^{1,r-1}(\overline B_R^+)}\leq C(R)\quad \mbox{for all sufficiently large }\al.
\ee
Therefore, $v_\al \to U$ in $C^1_{loc}(\overline{\R^{n}_+})$ for some $U\in C^{1,r-1}_{loc}(\overline{\R^{n}_+})\cap C^\infty(\R^{n}_+)$ which satisfies
\be\label{limt equ}
\begin{cases}
-\Delta U=0,\quad &\mbox{in }\mathbb{R}^{n}_+,\\
-\pa_{y_{n}}U=\frac{1}{S}U^{q-1},\quad & \mbox{on }\pa\mathbb{R}^{n}_+,\\
U(0)=1,\quad 0\leq U\leq 1.
\end{cases}
\ee
By the Liouville type theorem in \cite{LZ95},
\be\label{bub}
U(y',y_n)=\left(\frac{(n-2)^2S^2}{|y'|^2+(y_n+(n-2)S)^2}\right)^\frac{n-2}{2},
\ee
where $y'=(y_1,\cdots, y_{n-1})$. Denote $\lda_0=(n-2)S$ and
\[
 U_{x_0,\lda}(x',x_n):=\lda^{-\frac{n-2}{2}}U((x'-x_0')/\lda, x_n/\lda)
=\lda_0^{\frac{n-2}{2}} \left(\frac{\lda \lda_0}{|x'-x_0'|^2+(x_n+\lda\lda_0)^2}\right)^\frac{n-2}{2},
\]
where  $x_0\in \pa \R^{n}$, $\lda>0$. For brevity, we write $U_{\lda}$ as $U_{0,\lda}$. Hence, $U_{1}=U$.

\begin{prop}\label{prop:energy converges} For every $\delta>0$,
\[
 \lim_{\al \to \infty}\left(\int_{B_{\delta}^+}|\nabla_g(u_\al-U_{\mu_\al})|^2\,\ud v_g+\int_{\pa' B_\delta^+}|u_\al-U_{\mu_\al}|^q\,\ud s_g\right)=0.
\]
\end{prop}

\begin{proof} We only prove that $$\lim_{\al \to \infty}\int_{B_{\delta}^+}|\nabla_g(u_\al-U_{\mu_\al})|^2\,\ud v_g=0,$$ and the other can be proved similarly.

For every given $\va>0$, one can find $\al_0>0$  such that for all $\al \geq \al_0$, 
$$\int_{M}|\nabla_gu_\al|^2\,\ud v_g\leq \frac{1}{S}+\va,$$
because $\lim_{\al \to \infty}\int_{M}|\nabla_gu_\al|^2\,\ud v_g=\frac{1}{S}$. Since $ \int_{\R^{n}_+}|\nabla U|^2 =1/S$, we can choose $R>0$ such that
\[
 \int_{\R^n_+\setminus B_{R}^+}|\nabla U|^2 \leq \va.
\]
Note that
\[
\int_{B_{\delta/\mu_\al}^+}|\nabla_{g_\al}v_\al|^2\,\ud v_{g_\al}= \int_{B_{\delta}^+}|\nabla_gu_\al|^2\,\ud v_g\leq \frac{1}{S}+\va.
\]
Hence
\[
\begin{split}
 \int_{B_{\delta}^+}|\nabla_g(u_\al-U_{\mu_\al})|^2\,\ud v_g&= \int_{B_{\delta/\mu_\al}^+}|\nabla_{g_\al}(v_\al-U)|^2\,\ud v_{g_\al}\\&
=\int_{B_R^+}|\nabla_{g_\al}(v_\al-U)|^2\,\ud v_{g_\al}+\int_{B_{\delta/\mu_\al}^+\setminus \overline B_R^+}|\nabla_{g_\al}(v_\al-U)|^2\,\ud v_{g_\al}\\&
\leq 2\va+ 2\int_{B_{\delta/\mu_\al}^+\setminus \overline B_R^+}|\nabla_{g_\al}v_\al|^2\,\ud v_{g_\al}+2\int_{B_{\delta/\mu_\al}^+\setminus \overline B_R^+}|\nabla_{g_\al}U|^2\,\ud v_{g_\al}\\&
\leq 10 \va,
\end{split}
\]
where we used $\|v_\al-U\|_{C^1(\overline B_R^+)}\leq \va$ for large $\al$ and
\[
 \begin{split}
  &\int_{B_{\delta/\mu_\al}^+\setminus \overline B_R^+}|\nabla_{g_\al}v_\al|^2\,\ud v_{g_\al}+ \int_{B_{\delta/\mu_\al}^+\setminus \overline B_R^+}|\nabla_{g_\al}U|^2\,\ud v_{g_\al}\\
  &\quad\leq \frac{1}{S}+\va-\int_{B_R^+}|\nabla_{g_\al}v_\al|^2\,\ud v_{g_\al}+\frac{1}{S}+\va-\int_{B_R^+}|\nabla_{g_\al}U_\al|^2\,\ud v_{g_\al}\\
&\quad\leq 4\va.
 \end{split}
\]

\end{proof}

\begin{prop}\label{prop:uniform estimates}
For all large $\al$,
\[
 u_\al(x)\leq C\mu_\al^{\frac{n-2}{2}}\mathrm{dist}_g(x,x_\al)^{2-n}\quad \mbox{for all }x\in \overline{M},
\]
where $C>0$ depends only on $M, g$.
\end{prop}

Let
\[
 b_\al:=\begin{cases}
            \min\{\frac{n-2}{2}h_g+\al\left(\frac{\|u_\al\|_{L^r}}{u_\al}\right)^{2-r},1\},&\quad \mbox{if }u_\al \neq 0,\\
             1,&\quad \mbox{if }u_\al = 0.
           \end{cases}
\]
To prove Proposition \ref{prop:uniform estimates}, we need the following lemma about \emph{Neumann functions}.

\begin{lem}\label{lem:neumann f}
 There exists a unique weak solution $G_\al$ of
\be\label{neumann}
\begin{cases}
 -\Delta_g G_\al=0,& \quad \mbox{on }M,\\
\frac{\pa_g G_\al}{\pa \nu}+b_\al G_\al=\delta_{x_\al}, &\quad \mbox{on }\pa M,
\end{cases}
\ee
where $\delta_{x_\al}$ is the delta measure centered at $x_\al$. Moreover, $G_\al\in C^1_{loc}(\overline M\setminus\{x_\al\})$ and
\be\label{neumann behavior}
C^{-1}\mathrm{dist}_g(x,x_\al)^{2-n}\leq G_\al(x)\leq C\mathrm{dist}_g(x,x_\al)^{2-n}, \quad \mbox{for all }x\in \overline{M},
\ee
where $C>0$ depends only on $M, g$.
\end{lem}
\begin{proof}
We claim that
\be\label{eq:b-al 0}
 \lim_{\al \to \infty}vol_g\{b_\al <\frac{1}{2}\}=0.
\ee
Indeed, for every measurable set $E\subset \subset \pa M\cap\{u_\al >0\}$, we have
\[
 vol_g(E)=\int_{E}\,\ud s_g=\int_{E}u_\al^{r/2}u_\al^{-r/2}\,\ud s_g\leq \|u_\al\|^{r/2}_{L^r(E)}\|u_\al^{-1}\|^{r/2}_{L^r(E)}.
\]
It follows that
\[
 \begin{split}
  \|(\|u_\al\|_{L^r(\pa M)}u_\al^{-1})^{2-r}\|_{L^{r/(2-r)}(E)}&=\|u_\al\|^{2-r}_{L^r(\pa M)}\|u_\al^{-1}\|^{2-r}_{L^r(E)}\\
&\geq (vol_g(E))^{2(2-r)/r}.
 \end{split}
\]
Note that
\[
 \al (\|u_\al\|_{L^r(\pa M)}u_\al^{-1})^{2-r}< \frac{1}{2}(1+(n-2)|h_g|),
\]
if $b_\al <1/2$. Since $\{b_\al <\frac{1}{2}\}\subset\subset \pa M\cap\{u_\al >0\}$,
\[
 (vol_g\{b_\al <\frac{1}{2}\})^{2(2-r)/r}\leq \frac{C}{\al}.
\]
We verified the claim.

Notice that $b_\al$ is uniformly bounded and Lipschitz. Thus,
\[
\begin{split}
&\int_{M}|\nabla_g u|^2\ud v_g+\int_{\pa M} b_\al u^2\ud s_g\\
&\ge\int_{M}|\nabla_g u|^2\ud v_g+\frac{1}{2}\int_{\pa M} u^2\ud s_g-\int_{\pa M} (b_\al-\frac{1}{2})^- u^2\ud s_g\\
&\ge \int_{M}|\nabla_g u|^2\ud v_g+\frac{1}{2}\int_{\pa M} u^2\ud s_g- \|(b_\al-\frac{1}{2})^- \|_{L^{n-1}(\pa M)}\|u\|^2_{L^q(\pa M)}\\
&\ge \int_{M}|\nabla_g u|^2\ud v_g+\frac{1}{2}\int_{\pa M} u^2\ud s_g-C\|(b_\al-\frac{1}{2})^- \|_{L^{n-1}(\pa M)} (\int_{M}|\nabla_g u|^2\ud v_g+\int_{\pa M} u^2\ud s_g)\\
&\ge \frac{1}{4} \int_{M}|\nabla_g u|^2\ud v_g+\frac{1}{4}\int_{\pa M} u^2\ud s_g,
\end{split}
\]
where we have used Theorem 0.1 in \cite{LZ97} and \eqref{eq:b-al 0}. It follows that the first eigenvalue $\lda_{1,\al}$ of
\be
\begin{cases}
 -\Delta_g f=0,& \quad \mbox{on }M,\\
\frac{\pa_g f}{\pa \nu}+b_\al f=\lda_{1,\al}f, &\quad \mbox{on }\pa M
\end{cases}
\ee
is uniformly lower bounded by a positive number. Thus, by standard elliptic equation theory there exists a solution of \eqref{neumann} satisfying \eqref{neumann behavior}.
\end{proof}

\begin{proof}[Proof of Proposition \ref{prop:uniform estimates}] Set $\varphi_\al=\mu_\al^{(n-2)/2}G_\al$, $w_\al=u_\al/\varphi_\al$ and $\hat{g}=\varphi_\al^{\frac{4}{n-2}}g$.
By the conformal invariance (see, e.g., (1.8) in \cite{E92a}), it is direct to verify that for $\al$ large, $w_\al$ satisfies
\[
\begin{cases}
 \Delta_{\hat{g}}w_\al=0&\quad \mbox{on }M,\\
\frac{\pa_{\hat{g}}w_\al}{\pa \nu}=\ell_\al w_\al^{q-1}-w_\al\varphi_\al^{-\frac{n}{n-2}}(\frac{\pa_g \varphi_\al}{\pa \nu}+\frac{n-2}{2}h_g\varphi_\al+\al\|u_\al\|_{L^r}^{2-r}u_\al^{r-2}\varphi_\al) & \quad \mbox{on }\pa M\setminus\{x_\al\}.
\end{cases}
\]
By our choice of $G_\al$, we have
\[
\begin{cases}
 \Delta_{\hat{g}}w_\al=0&\quad \mbox{on }M,\\
\frac{\pa_{\hat{g}}w_\al}{\pa \nu}\leq \ell_\al w_\al^{q-1} & \quad \mbox{on }\pa M\setminus\{x_\al\}.
\end{cases}
\]
Then the Moser iterations procedure on page 465-471 of \cite{LZ97} implies that
\[
 \|w_\al\|_{L^\infty(M\setminus B_{\mu_\al}(x_\al))}\leq C.
\]
Recall that $v_\al\to U$ in $C^2_{loc}$, from which we also have $ \|w_\al\|_{L^\infty(B_{\mu_\al}(x_\al))}\leq C$. Proposition \ref{prop:uniform estimates} follows immediately.
\end{proof}

\begin{cor}
\label{cor:choose delta}
For any small $\delta>0$, there exists a constant $C>0$ depending only on $M,g, \delta$ such that
\[
\int_{M\setminus B_{\delta/2}(x_\al)}|\nabla_g u_\al|^2\leq C\mu_\al^{n-2},
\]
where $ B_{\delta/2}(x_\al)$ centered at $x_\al$ with radius $\delta/2$. Consequently, we can select $\delta_\al\in [\delta/2,\delta]$ such that
\[
\int_{\pa B_{\delta_\al}(x_\al)}|\nabla_g u_\al|^2\leq C\mu_\al^{n-2}.
\]
\end{cor}

\begin{proof}
Let $\eta$ be a cut-off function satisfying $Supp(\eta)\subset M\setminus B_{\delta/2}(x_\al)$ and  $\eta\equiv 1$ in  $M\setminus B_{\delta}(x_\al)$.
 Multiplying the Euler-Lagrange equation \eqref{el equ} by
$\eta^2u_\al$ and integrating by parts, we have
\[
  \int_{M} \nabla_gu_\al \nabla_g(\eta^2 u_\al)\,\ud v_g\leq -\frac{n-2}{2}\int_{\pa M}h\eta^2 u_\al^2\,\ud s_g +\ell_\al \int_{\pa M}\eta^2 u_\al^q\,\ud s_g.
\]
It follows that
\[
  \int_{M} \eta^2|\nabla_gu_\al|^2  \,\ud v_g\leq C\int_{\pa M}\Big(\eta^2 u_\al^2 +\eta^2 u_\al^q\Big)\,\ud s_g + C\int_{M} |\nabla_g \eta|^2 u_\al^2\,\ud v_g.
\]
Therefore, this corollary follows immediately from Proposition \ref{prop:uniform estimates}. 
\end{proof}

\section{Energy estimates and proof of Theorem \ref{thm:Tr}}

For some small $\delta_0$ to be determined in Lemma \ref{lem:positive def1}, let
 $\psi_\al\in C^\infty(\overline M)$ satisfy $\psi_\al(x_\al)=1$, $\frac{1}{2}\leq \psi_\al\leq 2$, $\|\psi_\al\|_{C^2(\overline{M})}\leq C$, and
\[
 \begin{cases}
  \Delta_{g}\psi_\al=0,&\quad \mbox{in }B_{2\delta_0}^+,\\
\frac{\pa_{g}\psi_\al}{\pa \nu}+\frac{n-2}{2}h_g\psi_\al=0, & \quad \mbox{on }\pa'B_{2\delta_0}^+.
 \end{cases}
\]
As in the previous section, here we used the Fermi coordinate with respect to metric $g$ centered at $x_\al$.
Set $\hat{g}=\psi_\al^{4/(n-2)}g$. It is easy to see that $h_{\hat{g}}=0$ on $\pa'B_{\delta_0}^+$. Hence, $u_\al/\psi_\al$ satisfies
\be\label{quotient eq}
 \begin{cases}
  \Delta_{\hat{g}}\frac{u_\al}{\psi_\al}=0,&\quad \mbox{in }B_{2\delta_0}^+,\\
\frac{\pa_{\hat{g}}}{\pa \nu}\frac{u_\al}{\psi_\al}+\al\|u_\al\|^{2-r}_{L^r(\pa M)}\psi_\al^{1-q}u_\al^{r-1}=\ell_\al(\frac{u_\al}{\psi_\al})^{q-1}, & \quad \mbox{on }\pa'B_{2\delta_0}^+.
 \end{cases}
\ee
It follows that the maximum of $u_\al/\psi_\al$ on $\overline{M}$ is achieved at some point on $\pa M$ which is denoted as $\hat{x}_\al$.
In view of the fact $u_\al(x_\al)\to \infty$ and Proposition \ref{prop:uniform estimates}, we have
$|\hat x_\al-x_\al| \to 0$ and thus $\frac{u_\al(\hat x_\al)}{\varphi_\al(\hat x_\al)u_\al(x_\al)}\to 1$ as $\al \to \infty$. From now on, we use the Fermi coordinate with respect to metric $\hat{g}$ centered at $x_\al$. Since $\psi_\al(x_\al)=1$, by simple blow-up argument and the same proof of Proposition \ref{prop:energy converges} we can establish
\be\label{ec1}
 \lim_{\al \to \infty}\left(\int_{B_{\delta}^+}|\nabla_{\hat{g}}(\frac{u_\al}{\psi_\al}-U_{\mu_\al})|^2\,\ud v_{\hat{g}}+\int_{\pa' B_\delta^+}|\frac{u_\al}{\psi_\al}-U_{\mu_\al}|^q\,\ud s_{\hat{g}}\right)=0.
\ee
As in Corollary \ref{cor:choose delta}, we can select $\delta_\al \in [\delta_0/2,\delta_0]$ such that
\be\label{fix del}
\int_{\pa''B_{\delta_\al}}|\nabla_{\hat{g}}(\frac{u_\al}{\psi_\al})|^2\,\ud v_{\hat{g}} \leq C \mu_\al^{n-2},
\ee
where $C>0$ is independent of $\al$.
Let us focus on the upper-half ball $B^+_{\delta_\al}(0)$, which is equipped with the Riemannian metric
$\hat g$.

Let $h_{\xi,\lda }(z)$, $z\in \overline B^+_{\delta_\al}(0)$, be the classical solution of
\[
\begin{cases}
 \Delta_{\hat g}h_{\xi,\lda }=0& \quad \mbox{in }B^+_{\delta_\al}, \\
h_{\xi,\lda }=U_{\xi,\lda}&\quad \mbox{on }\pa''B^+_{\delta_\al},\\
\frac{\pa_{\hat{g}}h_{\xi,\lda}}{\pa \nu}=0 &\quad \mbox{on }\pa'B^+_{\delta_\al},
\end{cases}
\]
with parameters $\xi\in \pa' B^+_{\delta_\al\mu_\al/2}$ and $\lda>0$,
while let $\chi_\al(z)$ be the solution of
\[
\begin{cases}
 \Delta_{\hat g}\chi_\al=0& \quad \mbox{in }B^+_{\delta_\al}, \\
\chi_\al=\frac{u_\al}{\psi_\al}&\quad \mbox{on }\pa''B_{\delta_\al},\\
\frac{\pa_{\hat{g}}\chi_\al}{\pa \nu}=0 &\quad \mbox{on }\pa'B^+_{\delta_\al}.
\end{cases}
\]
Then $u_\al/\psi_\al-\chi_\al\in H_{0,L}(B^+_{\delta_\al})$, $U_{\xi,\lda}-h_{\xi,\lda}\in H_{0,L}(B^+_{\delta_\al})$ are the projections of $u_\al/\psi_\al$ and
$U_{\xi,\lda}$ on $H_{0,L}(B_{\delta_\al})$, respectively. Here
\[
 H_{0,L}(B^+_{\delta_\al}):=\{u\in H^1(B^+_{\delta_\al}): u=0 \mbox{ on }\pa''B^+_{\delta_\al} \mbox{ in trace sense}\}
\]
is a Hilbert space with the inner product
\[
 \langle u, v\rangle_{\hat{g}}:=\int_{B^+_{\delta_\al}}\nabla_{\hat{g}}u\nabla_{\hat{g}} v\,\ud v_{\hat{g}}.
\]
Denote $\|u\|=\sqrt{\langle u, u\rangle_{\hat{g}}}$, which is a norm for $u\in H_{0,L}(B^+_{\delta_\al})$.

Set
\[
 \sigma_{\xi,\lda}=U_{\xi,\lda}-h_{\xi,\lda},
\]
which satisfies $\sigma_{\xi,\lda}\leq U_{\xi,\lda}$ and
\[
 \begin{cases}
 \Delta_{\hat g}\sigma_{\xi,\lda} = \Delta_{\hat g}U_{\xi,\lda},& \quad \mbox{in }B^+_{\delta_\al}, \\
\sigma_{\xi,\lda}=0,&\quad \mbox{on }\pa''B^+_{\delta_\al},\\
\frac{\pa_{\hat{g}}\sigma_{\xi,\lda}}{\pa \nu}=\frac{\pa_{\hat{g}}U_{\xi,\lda}}{\pa \nu} &\quad \mbox{on }\pa'B^+_{\delta_\al}.
\end{cases}
\]
Let $(t_\al,\xi_\al,\lda_\al)\in [\frac{1}{2},\frac{3}{2}]\times \overline{\pa' B^+_{\delta_\al\mu_\al/2}}\times [\frac{\mu_\al}{2},\frac{3\mu_\al}{2}]$ be such that
\[
 \begin{split}
  &\|\frac{u_\al}{\psi_\al}-\chi_\al-t_\al \sigma_{\xi_\al,\lda_\al}\|\\&
=\min\Big\{\|\frac{u_\al}{\psi_\al}-\chi_\al-t \sigma_{\xi,\lda}\|: (t,\xi,\lda)\in [\frac{1}{2},\frac{3}{2}]\times \overline{\pa' B^+_{\delta_\al\mu_\al/2}}\times [\frac{\mu_\al}{2},\frac{3\mu_\al}{2}] \Big\},
 \end{split}
\]
and
\[
 w_\al= \frac{u_\al}{\psi_\al}-\chi_\al-t_\al \sigma_{\xi_\al,\lda_\al}.
\]
Define
\[
 W_\al=\{w\in H_{0,L}(B^+_{\delta_\al}): \langle \sigma_{\xi_\al,\lda_\al}, w\rangle_{\hat{g}}=0 \mbox{ and } \langle u, w\rangle_{\hat{g}}=0 \mbox{ for all }u\in E_\al\},
\]
where $E_\al\subset H_{0,L}(B^+_{\delta_\al})$ is the tangent space at $\sigma_{\xi_\al,\lda_\al}$ of the $n$-dimensional surface
$\{\sigma_{\xi,\lda}: \xi\in \pa' B_{\mu_\al \delta_\al}, \lda >0\}\subset H_{0,L}(B_{\delta_\al}^+)$. More explicitly,
\[
 E_\al=\mathrm{span}\{\frac{\pa \sigma_{\xi,\lda_\al}}{\pa \xi} \Big|_{\xi=\xi_\al}, \frac{\pa \sigma_{\xi_\al,\lda}}{\pa \lda}\Big|_{\lda=\lda_\al}\}.
\]

For brevity, we denote $h_\al=h_{\xi_\al,\lda_\al}$ and $\sigma_\al=\sigma_{\xi_\al,\lda_\al}$

\begin{lem} \label{lem:small properties}
We have,
\begin{itemize}
  \item[1)] $\|\nabla_{\hat{g}} h_\al\|_{L^2(B_{\delta_\al})}+\|h_\al\|_{L^\infty(B_{\delta_\al})}\leq C\mu_\al^{(n-2)/2}$,
  \item[2)] $\|\nabla_{\hat{g}} \chi_\al\|_{L^2(B_{\delta_\al})}+\|\chi_\al\|_{L^\infty(B_{\delta_\al})}\leq C\mu_\al^{(n-2)/2}$,
\end{itemize}
 for some positive constant $C$ independent of $\al$, and
\begin{itemize}
  \item[3)] $\|w_\al\| \to 0$,
  \item[4)] $t_\al\to 1$,
  \item[5)] $\mu_\al^{-1}|\xi_\al|\to 0$,
  \item[6)] $\mu_\al^{-1}\lda_\al\to 1$,
\end{itemize}
as $\al\to \infty$. Furthermore, $w_\al\in W_\al$.
\end{lem}

\begin{proof} Let $\eta\in H^1(B_{\delta_\al})$ be an extension of $U_{\xi_\al,\lda_\al}|_{\pa''B^+_{\delta_\al}}$ such that
 \be\label{eq:choice of eta}
 \|\eta\|^2_{H^1(B_{\delta_\al})}\leq C\left(\int_{\pa''B^+_{\delta_\al}}|\nabla U_{\xi_\al,\lda_\al}|^2+U_{\xi_\al,\lda_\al}^2\right),
 \ee
 where $C>0$ is independent of  $\al$.
Multiplying the equation of $h_\al$ by $h_\al-\eta\in H_{0,L}(B_{\delta_{\al}})$ and integrating by parts we have
\[
0=\int_{B_{\delta_{\al}}}\nabla_{\hat{g}} h_\al\nabla_{\hat{g}}(h_\al-\eta)\,\ud v_{\hat{g}}\geq \frac12(\|\nabla_{\hat{g}} h_\al\|^2_{L^2}
-\|\nabla_{\hat{g}}\eta\|^2_{L^2}).
\]
Thus we obtained the $L^2$ estimate for $\nabla_{\hat{g}} h_\al$. The  $L^\infty$ estimates for $h_\al$ follows easily from the maximum principle.
Hence, we verified $1)$. Similarly, we can verify $2)$ by taking into account
Proposition \ref{prop:uniform estimates} and \eqref{fix del}.

By the definition of $t_\al$ and $\sigma_\al$,
\[
\begin{split}
\|t_\al\sigma_\al-\sigma_{0,\mu_\al}\|&\leq \|\frac{u_\al}{\psi_\al}-\chi_\al-t_\al\sigma_\al\|+\|\frac{u_\al}{\psi_\al}-\chi_\al-\sigma_{0,\mu_\al}\|\\
&\leq 2\|\frac{u_\al}{\psi_\al}-\chi_\al-\sigma_{0,\mu_\al}\| \\&
\leq 2\|\frac{u_\al}{\psi_\al}-U_{\mu_\al}\|+C\mu_\al^{(n-2)/2} \to 0
\end{split}
\]
as $\al \to \infty$, where we used $1)$, $2)$ and \eqref{ec1}. It follows that $\|w_\al\| \to 0$, i.e., $3)$ holds,  and
\[
\|t_\al U_{\xi_\al,\lda_\al}-U_{\mu_\al}\|\leq \|t_\al\sigma_\al-\sigma_{0,\mu_\al}\|+\|t_\al h_{\xi_\al,\lda_\al}\|+\|h_{0,\mu_\al}\|\to 0
\]
as $\al \to \infty$. A simple calculation yields $4)$, $5)$ and $6)$. Once we have $4)$, $5)$ and $6)$, the minimum of the norm is attained in the interior of
$ [\frac{1}{2},\frac{3}{2}]\times \overline{\pa' B^+_{\delta_\al\mu_\al/2}}\times [\frac{\mu_\al}{2},\frac{3\mu_\al}{2}]$. Hence, an  variational
argument gives $w_\al\in W_\al$.
\end{proof}

\begin{prop} \label{prop:energy est 1}
Assume as the above, we have
\[
\begin{split}
\|w_\al\|+|t_\al^{q-2}\ell_\al-S^{-1}|\leq C\Big\{&\mu_\al\|U^{\frac{n-1}{n-2}}_1\|_{L^{2^{*'}}(B^+_{\mu_\al^{-1}})}\\&
+\ep_\al \|U_1^{r-1}\|_{L^r(\pa' B^+_{\mu_\al^{-1}})}+\mu_\al^{n-2}\|U^{q-2}_1\|_{L^r(\pa' B^+_{\mu_\al^{-1}})}\Big\},
\end{split}
\]
where $\ep_\al$ is given in \eqref{eq:epsilon alpha}.
\end{prop}

We postpone the proof of Proposition \ref{prop:energy est 1} to the end of this section, and use it to prove Theorem \ref{thm:Tr} first.

Let
\[
Y(g,u_\al)= \frac{\int_{M}|\nabla_g u_\al|^2\,\ud v_g +\frac{n-2}{2}\int_{\pa M}h_g u_\al^2\,\ud s_g}{\Big(\int_{\pa M}u^q_\al\,\ud s_g\Big)^{2/q}}.
\]
It follows from Proposition \ref{prop:uniform estimates} that
\[
Y(g,u_\al)= \frac{\int_{B_{\delta_\al}^+}|\nabla_g u_\al|^2\,\ud v_g +\frac{n-2}{2}\int_{\pa' {B_{\delta_\al}^+}}h_g u_\al^2\,\ud s_g}{\Big(
\int_{\pa' {B_{\delta_\al}^+}}u^q_\al\,\ud s_g\Big)^{2/q}}+O(\mu_\al^{n-2}).
\]
Since
\[
\begin{cases}
\Delta_g u_\al-c(n)R_gu_\al&=\varphi_{\al}^{\frac{n+2}{n-2}}(\Delta_{\hat{g}} \frac{u_\al}{\varphi_\al}-c(n)R_{\hat{g}} \frac{u_\al}{\varphi_\al})\quad \mbox{on }M, \\
\frac{\pa_g u_\al}{\pa \nu}+\frac{n-2}{2}h_gu_\al&=\varphi_{\al}^{\frac{n}{n-2}}(\frac{\pa_{\hat{g}}}{\pa \nu}\frac{u_\al}{\varphi_\al}
+\frac{n-2}{2}h_{\hat{g}} \frac{u_\al}{\varphi_\al} ) \quad \mbox{on }\pa M,
\end{cases}
\]
and $h_{\hat{g}}=0$ on $\pa' {B_{\delta_\al}^+}$, we have
\[
Y(g,u_\al)= \frac{\int_{B_{\delta_\al}^+}|\nabla_{\hat{g}} (\frac{u_\al}{\varphi_\al} )|^2\,\ud v_{\hat{g}}}{\Big(
\int_{\pa' {B_{\delta_\al}^+}}(\frac{u_\al}{\varphi_\al})^q\,\ud s_{\hat{g}}\Big)^{2/q}}+O(\mu_\al^{n-2}).
\]
In view of $u_\al/\varphi_\al=t_\al U_{\xi_\al,\lda_\al} -t_\al h_\al+\chi_\al+w_\al$,
\[
\int_{B^+_{\delta_\al}}\nabla_{\hat{g}}\chi_\al \nabla_{\hat{g}} w_\al\,\ud v_{\hat{g}}= \int_{B^+_{\delta_\al}}\nabla_{\hat{g}}h_\al \nabla_{\hat{g}} w_\al\,\ud v_{\hat{g}}=0,
\]
and the estimates in Lemma \ref{lem:small properties}, we have
\be\label{Y reduce}
Y(g,u_\al)=F(w_\al)+O(\mu_\al^{n-2}),
\ee
where
\[
F(w):= \frac{\int_{B_{\delta_\al}^+}|\nabla_{\hat{g}} (t_\al U_{\xi_\al,\lda_\al}+w)|^2\,\ud v_{\hat{g}}}{\Big(
\int_{\pa' {B_{\delta_\al}^+}}|t_\al U_{\xi_\al,\lda_\al}+w|^q\,\ud s_{\hat{g}}\Big)^{2/q}}.
\]
By a direct computation,
\[
\begin{split}
F'(0)w_\al=&\frac{2}{\Big(\int_{\pa' B_{\delta_\al}^+}|t_\al U_{\xi_\al,\lda_\al}|^q\,\ud s_{\hat{g}}\Big)^{2/q}}\Big\{
\int_{B_{\delta_\al}^+}t_\al\nabla_{\hat{g}}U_{\xi_\al,\lda_\al} \nabla_{\hat{g}} w_\al\,\ud v_{\hat{g}} \\& -
\frac{\int_{B_{\delta_\al}^+}|t_\al\nabla_{\hat{g}}U_{\xi_\al,\lda_\al}|^2\,\ud v_{\hat{g}}}{\int_{\pa' B_{\delta_\al}^+}|t_\al U_{\xi_\al,\lda_\al}|^q\,\ud s_{\hat{g}}} \int_{\pa' B_{\delta_\al}^+}|t_\al U_{\xi_\al,\lda_\al}|^{q-1}w_\al\,\ud s_{\hat{g}} \Big\}.
\end{split}
\]
Recall that $\sigma_\al=U_{\xi_\al,\lda_\al}-h_\al$ and $\int_{B_{\delta_\al}^+}\nabla_{\hat{g}} \sigma_\al \nabla_{\hat{g}} w_\al=\int_{B_{\delta_\al}^+}\nabla_{\hat{g}} h_\al \nabla_{\hat{g}} w_\al=0$. It follows that $\int_{B_{\delta_\al}^+}\nabla_{\hat{g}} U_{\xi_\al,\lda_\al} \nabla_{\hat{g}} w_\al=0$. Hence, (recall that we are in the Fermi coordinates of $\hat g$),
\[
\begin{split}
|F'(0)w_\al|& \leq C \Big|\int_{\pa' B_{\delta_\al}^+}| U_{\xi_\al,\lda_\al}|^{q-1}w_\al\,\ud s_{\hat{g}} \Big|
=C S\Big|\int_{\pa' B_{\delta_\al}^+}\frac{\pa_{\hat{g}}}{\pa \nu} U_{\xi_\al,\lda_\al} w_\al\,\ud s_{\hat{g}} \Big|\\&
=CS \Big|\int_{B_{\delta_\al}^+} \Delta_{\hat{g}}U_{\xi_\al,\lda_\al} w_\al\,\ud v_{\hat{g}} \Big|\\&
\leq C \|\Delta_{\hat{g}}U_{\xi_\al,\lda_\al}\|_{L^{2^{*'}}(B_{\delta_\al}^+)}\|w_\al\|,
\end{split}
\]
where $2^*=\frac{2n}{n-2}$ and $2^{*'}=\frac{2n}{n+2}$. 
By \eqref{lap bound} we have
\[
|F'(0)w_\al| \leq C\mu_\al  \|U_1^{\frac{n-1}{n-2}}\|_{L^{2^{*'}}(B^+_{\mu_\al^{-1}})} \|w_\al\|.
\]
Similarly,
\[
\begin{split}
\langle F''(0)w_\al, w_\al\rangle =&\frac{2}{\Big(\int_{\pa' B_{\delta_\al}^+}|t_\al U_{\xi_\al,\lda_\al}|^q\,\ud s_{\hat{g}}\Big)^{2/q}}\Big\{
\int_{B_{\delta_\al}^+}|\nabla_{\hat{g}} w_\al|^2\,\ud v_{\hat{g}} \\& -(q-1)
\frac{\int_{B_{\delta_\al}^+}|t_\al\nabla_{\hat{g}}U_{\xi_\al,\lda_\al}|^2\,\ud v_{\hat{g}}}{\int_{\pa' B_{\delta_\al}^+}|t_\al U_{\xi_\al,\lda_\al}|^q\,\ud s_{\hat{g}}} \int_{\pa' B_{\delta_\al}^+}|t_\al U_{\xi_\al,\lda_\al}|^{q-2}w_\al^2\,\ud s_{\hat{g}} \Big\}\\&
+O\left(\Big( \int_{\pa' B_{\delta_\al}^+}| U_{\xi_\al,\lda_\al}|^{q-1}w_\al\,\ud s_{\hat{g}}\Big)^2\right).
\end{split}
\]
By Lemma \ref{lem:comanifolds}, we have
\[
\begin{split}
&\int_{B_{\delta_\al}^+}|\nabla_{\hat{g}} w_\al|^2\,\ud v_{\hat{g}} -
\frac{(q-1)\int_{B_{\delta_\al}^+}|t_\al\nabla_{\hat{g}}U_{\xi_\al,\lda_\al}|^2\,\ud v_{\hat{g}}}{\int_{\pa' B_{\delta_\al}^+}|t_\al U_{\xi_\al,\lda_\al}|^q\,\ud s_{\hat{g}}} \int_{\pa' B_{\delta_\al}^+}|t_\al U_{\xi_\al,\lda_\al}|^{q-2}w_\al^2\,\ud s_{\hat{g}}\\
&\quad\quad\geq \frac{c_1}{2}\|w_\al\|^2,
\end{split}
\]
for large $\al$. It follows that
\[
\langle F''(0)w_\al, w_\al\rangle\geq \frac{c_1}{2}\|w_\al\|^2+O(\mu_\al^2) \|U_1^{\frac{n-1}{n-2}}\|^2_{L^{2^{*'}}(B^+_{\mu_\al^{-1}})} \|w_\al\|^2.
\]
Noticing that $\mu_\al \|U_1^{\frac{n-1}{n-2}}\|^2_{L^{2^{*'}}(B^+_{\mu_\al^{-1}})}\to 0$ as $\al \to \infty$, we have
\[
\begin{split}
F(w_\al)&= F(0)+F'(0)w_\al+\frac{1}{2}\langle F''(0)w_\al, w_\al\rangle+o(\|w_\al\|^2)\\&
\geq F(0)+O\Big(\mu_\al  \|U_1^{\frac{n-1}{n-2}}\|_{L^{2^{*'}}(B^+_{\mu_\al^{-1}})} \|w_\al\|\Big).
\end{split}
\]
By \eqref{Y reduce}, we conclude that
\be \label{lower bounde of Y}
Y(g, u_\al)\geq \frac{\int_{B_{\delta_\al}^+}|\nabla_{\hat{g}} U_{\xi_\al,\lda_\al}|^2\,\ud v_{\hat{g}}}{ \Big(
\int_{\pa' {B_{\delta_\al}^+}} U_{\xi_\al,\lda_\al}^q\,\ud s_{\hat{g}}\Big)^{2/q}}+O\Big(\mu_\al  \|U_1^{\frac{n-1}{n-2}}\|_{L^{2^{*'}}(B^+_{\mu_\al^{-1}})} \|w_\al\|+\mu_\al^{n-2}\Big).
\ee

\begin{lem} \label{lem:buble check} We have
\[
\frac{\int_{B_{\delta_\al}^+}|\nabla_{\hat{g}} U_{\xi_\al,\lda_\al}|^2\,\ud v_{\hat{g}}}{ \Big(
\int_{\pa' {B_{\delta_\al}^+}} U_{\xi_\al,\lda_\al}^q\,\ud s_{\hat{g}}\Big)^{2/q}}=\begin{cases}
\frac{1}{S} +O(\mu_\al^2),& \quad n\geq 5,\\[4mm]
\frac{1}{S} +O(\mu_\al^2\log \mu_\al^{-1}), &\quad n=4.
\end{cases}
\]
\end{lem}

\begin{proof} Let $\pi=\hat{h}_{ij} \ud z^i\ud z^j$, $1\leq i,j\leq n-1$,
be the second fundamental form of $ \pa' B_{\delta_\al}^+$ with respect to the metric $\hat{g}_{ij}$.
Since 
$\sqrt{\det \hat{g}_{ij}}=1+ O(|z|^2)$ on $\pa' {B_{\delta_\al}^+}$, we have
\[
\begin{split}
\Big(\int_{\pa' {B_{\delta_\al}^+}} U_{\xi_\al,\lda_\al}^q\,\ud s_{\hat{g}}\Big)^{2/q}&=\Big(\int_{\pa' {B_{\delta_\al}^+}} U_{\xi_\al,\lda_\al}^q
(1+O(|z'|^2))\,\ud z\Big)^{2/q}\\&
= \Big(\int_{\pa' {B_{\delta_\al}^+}} U_{\xi_\al,\lda_\al}^q\,\ud z\Big)^{2/q}+O(\int_{\pa' {B_{\delta_\al}^+}} U_{\xi_\al,\lda_\al}^q|z'|^2\,\ud z)\\&
= \Big(\int_{\pa \R^{n}_+} U_{\xi_\al,\lda_\al}^q\,\ud z\Big)^{2/q}+O(\lda^{n-2}_\al)\\&
\quad +O(\lda_\al^{2}\int_{\pa' {B_{\lda_\al^{-1}}^+}} (1+|z'|^2)^{1-n}|z'|^2\,\ud z)\\&
=1+O(\lda^{n-2}_\al)+O(\lda_\al^{2}\int_{\pa' {B_{\lda_\al^{-1}}^+}} (1+|z'|^2)^{1-n}|z'|^2\,\ud z).
\end{split}
\]
where we used $\lda^{-1}_\al|\xi_\al|\to 0$ as $\al \to \infty$ in the last equality.

In addition, by Lemma \ref{taylor expansion of fermi}, we have
\[
\begin{split}
\int_{B_{\delta_\al}^+}|\nabla_{\hat{g}} U_{\xi_\al,\lda_\al}|^2\,\ud v_{\hat{g}}&=\int_{B_{\delta_\al}^+}|\nabla U_{\xi_\al,\lda_\al}|^2\,\ud z
+2\hat{h}^{ij}(0) \int_{B_{\delta_\al}^+}  \pa_iU_{\xi_\al,\lda_\al}\pa_jU_{\xi_\al,\lda_\al} z_n \,\ud z\\&
\quad +O(\int_{B_{\delta_\al}^+}|\nabla U_{\xi_\al,\lda_\al}|^2|z|^2\,\ud z).
\end{split}
\]
It is easy to see that
\[
\int_{B_{\delta_\al}^+}|\nabla U_{\xi_\al,\lda_\al}|^2\,\ud z=\int_{\R^n_+}|\nabla U_{\xi_\al,\lda_\al}|^2\,\ud z+O(\lda_\al^{n-2})
 =\frac1S+O(\lda_\al^{n-2})
\]
and
\[
\begin{split}
\int_{B_{\delta_\al}^+}&|\nabla U_{\xi_\al,\lda_\al}|^2|z|^2\,\ud z \\&=C(n)\lda^2_\al\int_{B_{\lda_\al^{-1}}^+}(|z'|^2+(z_n+1)^2)^{1-n}|z+
\lda^{-1}_\al \xi_\al|^2\,\ud z+O(\lda_\al^{n-2}).
\end{split}
\]
By symmetry, we have
\[
\begin{split}
&\sum_{i,j=1}^{n-1}\hat{h}^{ij}(0) \int_{B_{\delta_\al}^+} \pa_iU_{\xi_\al,\lda_\al}\pa_jU_{\xi_\al,\lda_\al} z_n \,\ud z\\
&=\sum_{i,j=1}^{n-1}\hat{h}^{ij}(0) \int_{B_{\delta_\al/2}^+(\xi_\al)}  \pa_iU_{\xi_\al,\lda_\al}\pa_jU_{\xi_\al,\lda_\al} z_n \,\ud z+O(\mu_\al^{n-2})\\
&=\sum_{i=1}^{n-1}\hat{h}^{ii}(0)\int_{B_{\delta_\al/2}^+(\xi_\al)}|\pa_1U_{\xi_\al,\lda_\al}|^2 z_n \,\ud z +O(\mu_\al^{n-2})
=O(\mu_\al^{n-2}),
\end{split}
\]
where we used $\sum_{i=1}^{n-1}\hat{h}^{ii}(0)=0$ since the mean curvature of $\hat g$ is vanishing at $x_\al$. The lemma follows immediately from Lemma \ref{lem:small properties}.
\end{proof}

\begin{proof}[Proof of Theorem \ref{thm:Tr}] Notice that
\[
\frac{1}{S}>I_\al(u_\al)=Y(g,u_\al)+\al\|u_\al\|^2_{L^r(\pa M)}.
\]
By \eqref{lower bounde of Y} and Lemma \ref{lem:buble check}, we have
\[
\al\|u_\al\|^2_{L^r(\pa M)} =O(\mu_\al^2) +O\Big(\mu_\al  \|U_1^{\frac{n-1}{n-2}}\|_{L^{2^{*'}}(B^+_{\mu_\al^{-1}})} \|w_\al\|+\mu_\al^{n-2}\Big).
\]
By Proposition \ref{prop:energy est 1}, we find
\be\label{eq:final contra}
\begin{split}
\al\|u_\al\|^2_{L^r(\pa M)} &\leq C\Big\{\mu_\al  \|U_1^{\frac{n-1}{n-2}}\|_{L^{2^{*'}}(B^+_{\mu_\al^{-1}})} \Big(\mu_\al\|U^{\frac{n-1}{n-2}}_1\|_{L^{2^{*'}}(B^+_{\mu_\al^{-1}})}\\&
+\ep_\al \|U_1^{r-1}\|_{L^r(\pa' B^+_{\mu_\al^{-1}})}+\mu_\al^{n-2}\|U^{q-2}_1\|_{L^r(\pa' B^+_{\mu_\al^{-1}})}\Big)+\mu_\al^2\Big\}.
\end{split}
\ee
Due to $n\geq 5$, we have
\[
\begin{split}
\|U_1^{\frac{n-1}{n-2}}\|_{L^{2^{*'}}(B^+_{\mu_\al^{-1}})}  \leq &~ C\\
\mu_\al\|U_1^{r-1}\|_{L^r(\pa' B^+_{\mu_\al^{-1}})}\leq &~ C\mu_\al(1+\mu_\al^{\frac{n^2-8n+8}{2n}}) =o(1)\\
\|U^{q-2}_1\|_{L^r(\pa' B^+_{\mu_\al^{-1}})} \leq &~ C(1+\mu_\al^{(4-n)/2}).
\end{split}
\]
From \eqref{etimate for ep}, i.e., $\ep_\al\leq \al\|u_\al\|^2_{L^r(\pa M)}$,  it follows that
\[
\al\|u_\al\|^2_{L^r(\pa M)} \leq C \mu_\al^2.
\]
On the other hand,
\[
\|u_\al\|_{L^r(\pa M)}\geq \|u_\al\|_{L^r(B^+_{\mu_\al}(x_\al)\cap \pa M)}\geq C^{-1}\mu_\al \|v_\al\|_{L^r(\pa'  B^+_{1})}\geq C^{-1} \mu_\al,
\]
where $v_\al$ is as in \eqref{eq 0}. Hence
\[
\al \leq C.
\]
This is a contradiction.
\end{proof}

\subsection{Proof of Proposition \ref{prop:energy est 1}}

The rest of the paper is devoted to the proof of Proposition \ref{prop:energy est 1}. We start with the equation which $w_\al$ satisfies.

\begin{lem}
 \label{lem: w equ}
$w_\al$ satisfies
\be\label{w equ}
\begin{cases}
 -\Delta_{\hat{g}}w_\al=t_\al \Delta_{\hat{g}}U_{\xi_\al, \lda_\al},& \quad \mbox{in }B^+_{\delta_\al}, \\[2mm]
\frac{\pa_{\hat{g} w}}{\pa \nu}-\ell_\al(q-1)|\Theta_{\al}|^{q-3}\Theta_\al w_\al +b'|\Theta_\al|^{q-3}w_\al^2+b''|w_\al|^{q-1}=f_\al, &\quad \mbox{on }\pa'B^+_{\delta_\al},
\end{cases}
\ee
where
\[
 \begin{split}
  &\Theta_\al=t_\al \sigma_\al+\chi_\al \\
&f_\al=t_\al(\ell_\al t^{q-2}_\al-\frac{1}{S})U_{\xi_\al,\lda_\al}^{q-1}- \al\|u_\al\|^{2-r}_{L^r(\pa M)}\psi_\al^{1-q}u_\al^{r-1}+O(\mu_\al^{(n-2)/2}) U_{\xi_\al,\lda_\al}^{q-2},
 \end{split}
\]
and $b', b''$ are bounded functions with $b'\equiv 0$ if $n\geq 4$.
\end{lem}

\begin{proof}
The proof follows from straightforward computations. First of all, by the definition of $w_\al$ and by the equation \eqref{quotient eq} of $u_\al/\psi_\al$, we  have
\[
 -\Delta_{\hat{g}} w_\al=-\Delta_{\hat{g}}(\frac{u_\al}{\psi_\al}-\chi_\al-t_\al\sigma_\al)=t_\al \Delta_{\hat{g}}\sigma_\al= t_\al \Delta_{\hat{g}}U_{\xi_\al,\lda_\al} \quad
\mbox{in }B_{\delta_\al}^+,
\]
 and
\[
 \begin{split}
  \frac{\pa_{\hat{g}}w_\al}{\pa \nu}&=\ell_\al(\Theta_\al +w_\al)^{q-1}-\al\|u_\al\|^{2-r}_{L^r(\pa M)}\psi_\al^{1-q}u_\al^{r-1}-t_\al \frac{\pa_{\hat{g}}\sigma_\al}{\pa \nu}\\
 &=\ell_\al(\Theta_\al +w_\al)^{q-1}-\al\|u_\al\|^{2-r}_{L^r(\pa M)}\psi_\al^{1-q}u_\al^{r-1}-t_\al \frac{\pa_{\hat{g}} U_{\xi_\al,\lda_\al}}{\pa \nu}\\
&=\ell_\al (\Theta_\al +w_\al)^{q-1}-\al\|u_\al\|^{2-r}_{L^r(\pa M)}\psi_\al^{1-q}u_\al^{r-1}-\frac{t_\al}{S} U_{\xi_\al,\lda_\al}^{q-1},
 \end{split}
\]
where we used that $\frac{\pa_{\hat{g}} U_{\xi_\al,\lda_\al}}{\pa \nu}=\frac{\pa U_{\xi_\al,\lda_\al}}{\pa x_n}=U_{\xi_\al,\lda_\al}^{q-1}$ in Fermi coordinate systems.
Note that we have the following elementary expansion
\[
 (k+l)^{q-1}=|k|^{q-2}k+(q-1)|k|^{q-3}kl+b'(k,l)|k|^{q-3}l^2+b''(k,l)|l|^{q-1},
\]
for all $k,l$ such that $k+l>0$, where $b',b''$ are bounded and $b'\equiv 0$ if $n\geq 4$.
For $k=\Theta_\al$, $l=w_\al$, we obtain
\[
\begin{split}
 (\Theta_\al+w_\al)^{q-1}=&|\Theta_\al|^{q-2}\Theta_\al+(q-1)|\Theta_\al|^{q-3}\Theta_\al w_\al\\&
+b'|\Theta_\al|^{q-3}w_\al^2+b''|w_\al|^{q-1},
\end{split}
\]
And Lemma \ref{lem: w equ} follows from
\[
  \begin{split}
   |\Theta_\al|^{q-2}\Theta_\al&= |t_\al U_{\xi_\al,\lda_\al}-t_\al h_\al+\chi_\al|^{q-2}(t_\al U_{\xi_\al,\lda_\al}-t_\al h_\al+\chi_\al)\\&
=(t_\al U_{\xi_\al,\lda_\al })^{q-1}- (t_\al U_{\xi_\al,\lda_\al })^{q-2}(t_\al h_\al -\chi_\al)\\& \quad +
|t_\al U_{\xi_\al,\lda_\al}-\theta (t_\al h_\al+\chi_\al)|^{q-3}(t_\al h_\al+\chi_\al )(t_\al U_{\xi_\al,\lda_\al}-t_\al h_\al+\chi_\al)\\&
= (t_\al U_{\xi_\al,\lda_\al })^{q-1}+O(\mu_\al^{(n-2)/2}U_{\xi_\al,\lda_\al }^{q-2}),
  \end{split}
\]
where $\theta\in (0,1)$ and we have used $1)$ and $2)$ in Lemma \ref{lem:small properties}.
\end{proof}

Define
\[
 Q_\al(\varphi, \phi)=\int_{B^+_{\delta_\al}}\nabla_{\hat{g}}\varphi\nabla_{\hat{g}}\phi\,\ud v_{\hat{g}}-\ell_\al (q-1)\int_{\pa' B^+_{\delta_\al}}
|\Theta_\al|^{q-3}\Theta_\al \varphi\phi\,\ud s_{\hat{g}},
\]
for all $\varphi, \phi \in H_{0,L}(B^+_{\delta_\al})$.

\begin{lem} \label{lem:positive def1}
 There exist $0<\delta_0<<1$, $\al_0\geq 1$ and  $c_0>0$ independent of $\al$ such that
\[
 Q_\al(w,w)\geq c_0\int_{B^+_{\delta_\al}}|\nabla_{\hat{g}}w|^2\,\ud v_{\hat{g}}, \quad \forall ~ w\in W_\al, ~\al \geq \al_0.
\]
\end{lem}
\begin{proof}
It will follow from Lemma \ref{lem:small properties} and Lemma \ref{lem:comanifolds}. 
\end{proof}

\begin{proof}[Proof of Proposition \ref{prop:energy est 1}] Multiplying the both sides of \eqref{w equ} by $w_\al$, we arrive at
\[
 Q_{\al}(w_\al,w_\al)+o(\|w_\al\|^2)=t_\al \int_{B^+_{\delta_\al}} w_\al \Delta_{\hat{g}}U_{\xi_\al,\lda_\al}\,\ud v_{\hat{g}}+\int_{\pa' B^+_{\delta_\al}}w_\al f_\al \,\ud s_{\hat{g}},
\]
where we used
\[
\int_{\pa' B^+_{\delta_\al}} b''|w_\al|^q\,\ud s_{\hat{g}}\leq C\|\nabla_{\hat{g}}w_\al\|^q_{L^2(B^+_{\delta_\al})}=C\|w_\al\|^2\|w_\al\|^{q-2},
\]
and
\[
 \int_{\pa' B^+_{\delta_\al}} b'|\Theta_\al|^{q-3}|w_\al|^3\leq C\|w_\al\|^3
\]
if $n=3$ and using Lemma \ref{lem:small properties}. By the H\"older inequality, Sobolev inequality, and Lemma \ref{lem:positive def1}, we have
\[
 \|w\|\leq  C\Big\{ \|\Delta_{\hat{g}}U_{\xi_\al, \lda_\al}\|_{L^{2^{*'}}(B^+_{\delta_\al})}+\al \|u_\al\|_{L^r(\pa M)}^{2-r}\|u_\al^{r-1}\|_{L^{q'}(\pa'B^+_{\delta_\al})}+
\mu_\al^{\frac{n-2}{2}}\|U_{\xi_\al, \lda_\al}^{q-2} \|_{L^{q'}(\pa'B^+_{\delta_\al})}\Big\}
\]
with $q'=r=q/(q-1)$.
Since $\Delta U_{\xi_\al, \lda_\al}=0$, a direct computation in Fermi coordinate system yields
\be\label{lap bound}
\begin{split}
& \|\Delta_{\hat{g}}U_{\xi_\al, \lda_\al}\|_{L^{2^{*'}}(B^+_{\delta_\al})}\\
 &=\|\Delta_{\hat{g}}U_{\xi_\al, \lda_\al}- \Delta U_{\xi_\al, \lda_\al} \|_{L^{2^{*'}}(B^+_{\delta_\al})}\\
&\leq \|(\hat{g}^{ij}-\delta^{ij})\frac{\pa^2 U_{\xi_\al, \lda_\al}}{\pa z_i\pa z_j}-\hat{g}^{ij}\hat{\Gamma}_{ij}^k\frac{\pa u}{\pa z_k}\|_{L^{2^{*'}}(B^+_{\delta_\al})}\\
&\leq C\Big(\int_{B^+_{\delta_\al}}\big(|z|\lda^{(n-2)/2}_\al(|z'-\xi_\al'|^2+(z_n+\lda_0\lda_\al)^2)^{-n/2}\big)^{2^{*'}}\,\ud z\Big)^{1/2^{*'}} \\
&+ C\Big(\int_{B^+_{\delta_\al}}\big(\lda^{(n-2)/2}_\al(|z'-\xi_\al'|^2+(z_n+\lda_0\lda_\al)^2)^{-(n-1)/2}\big)^{2^{*'}}\,\ud z\Big)^{1/2^{*'}} \\
&\leq C\lda_\al\Big(\int_{B^+_{\lda_\al^{-1}}}\big(|y+\lda_\al^{-1}\xi_\al|(|y'|^2+(y_n+\lda_0)^2)^{-n/2}\big)^{2^{*'}}\,\ud y\Big)^{1/2^{*'}} \\
&+ C\lda_\al\Big(\int_{B^+_{\lda_\al^{-1}}}\big((|y'-\lda_\al^{-1}\xi_\al|^2+(y_n+\lda_0)^2)^{-(n-1)/2}\big)^{2^{*'}}\,\ud y\Big)^{1/2^{*'}} \\
&\leq C\lda_\al \|U_1^{\frac{n-1}{n-2}}\|_{L^{2^{*'}}(B^+_{\lda_\al^{-1}})},
\end{split}
\ee
where $2^*=\frac{2n}{n-2}$, $2^{*'}=\frac{2n}{n+2}$ and we used the variables transformation $y=\lda_\al^{-1}(z-\xi_\al)$ and the fact $\lda_\al^{-1}|\xi_\al|\to 0$ as $\al \to \infty$.
By Proposition \ref{prop:uniform estimates}, we have $u_\al \leq CU_{\mu_\al}$. Hence
\[
\|u_\al^{r-1}\|_{L^{q'}(\pa'B^+_{\delta_\al})}\leq C\|U_{\mu_\al}^{r-1}\|_{L^{q'}(\pa'B^+_{\delta_\al})}\leq \mu_\al^{n-1-\frac{n-2}{2}r}\|U_1^{r-1}\|_{L^{q'}(\pa'B^+_{\mu^{-1}_\al})}.
\]
Finally, it is clear that
\[
\|U_{\xi_\al, \lda_\al}^{q-2} \|_{L^{q'}(\pa'B^+_{\delta_\al})}\leq C\mu_\al^{\frac{n-2}{2}}
\|U_{1}^{q-2} \|_{L^{q'}(\pa'B^+_{\mu_\al^{-1}})}.
\]
Therefore, we obtained the estimate for $\|w_\al\|$.

Multiplying \eqref{w equ} by $\sigma_\al$ and integrating over $B_{\delta_\al}^+$, we find that
\[
\begin{split}
|t_\al(\ell_\al t_{\al}^{q-2}-\frac{1}{S})|&\int_{\pa' B_{\delta_\al}^+}U^{q-1}_{\xi_\al,\lda_\al}\sigma_\al\,\ud s_{\hat{g}}\\
\leq & C\Big( \int_{B_{\delta_\al}^+} |\Delta_{\hat{g}}U_{\xi_\al,\lda_\al} \sigma_\al|\,\ud v_{\hat{g}}+\int_{\pa' B_{\delta_\al}^+}U^{q-1}_{\xi_\al,\lda_\al}|w_\al|\,\ud s_{\hat{g}}\\&
+\al\|u_\al\|_{L^r(\pa M)}^{2-r}\int_{\pa' B_{\delta_\al}^+}U^{r}_{\xi_\al,\lda_\al}\,\ud s_{\hat{g}} +\mu_\al^{(n-2)/2}
\int_{\pa' B_{\delta_\al}^+}U^{q-1}_{\xi_\al,\lda_\al}\,\ud s_{\hat{g}}\Big),
\end{split}
\]
where we used $|\sigma_\al|+|\Theta_\al|+|w_\al|+u_\al\leq CU_{\xi_\al,\lda_\al}$. Since $\int_{\pa' B_{\delta_\al}^+}U^{q-1}_{\xi_\al,\lda_\al}\sigma_\al\,\ud s_{\hat{g}}\geq C^{-1}>0$, the estimate of $|\ell_\al t_{\al}^{q-2}-\frac{1}{S}|$ follows by H\"older inequality and the estimate for $\|w_\al\|$.
\end{proof}

\subsection{Proof of Lemma \ref{lem:positive def1}}

Let  $\mathcal{D}^{1,2}(\R^n_+)$ be the closure of $C_c^\infty(\R^n_+\cup \pa \R^n_+)$ under the norm
\[
 \|u\|_{\mathcal{D}^{1,2}(\R^n_+)}=\left(\int_{\R^n_+}|\nabla u|^2\,\ud y\right)^{1/2}.
\]
In fact, $\mathcal{D}^{1,2}(\R^n_+)$ is a Hilbert space with the inner product
$$\langle \varphi,\phi\rangle_1:=\int_{\R^n_+}\nabla \varphi \nabla \phi\,\ud y$$ for any $\varphi,\phi\in \mathcal{D}^{1,2}(\R^n_+)$. Define the functional
\[
 Q_1(\varphi, \phi):=\int_{\R^n_+}\nabla \varphi \nabla \phi\,\ud y-\frac{q-1}{S}\int_{\pa \R^n_+} U_{1}^{q-2}\varphi \phi\,\ud y'
\]
for all $\varphi, \phi\in E_0$,  where $y=(y',0)$ and
\[
\begin{split}
 E_1=\Big\{w\in \mathcal{D}^{1,2}(\R^n_+): &\quad  \langle \frac{\pa U_{\xi,1}}{\pa \xi_i}\Big|_{\xi=0}, w\rangle_1= \langle \frac{\pa U_{0,\lda}}{\pa \lda}\Big|_{\lda=1}, w\rangle_1
= \langle U_{0,1}, w\rangle_1=0,\\& i=1,2,\cdots, n-1 \Big\}.
\end{split}
\]
Then we have
\begin{lem}
\label{lem:cowhole}
There exists a constant $c_1>0$ depending only $n$ such that
\[
 Q_1(w, w)\geq  c_1\|w\|^2_{\mathcal{D}^{1,2}(\R^n_+) },
\]
for all $w\in E_1$.
\end{lem}
Lemma \ref{lem:cowhole} follows from an analysis of the eigenvalues of
\[
 \begin{cases}
  \Delta v=0,& \quad \mbox{in }B_1,\\
\frac{\pa v}{\pa \nu}=\lda v,&\quad  \mbox{on }\pa B_1.
 \end{cases}
 \]
The details can be found in, e.g., \cite{Del2}.
\begin{lem}
\label{lem:comanifolds}
For $R>0$, $x\in \pa \R^n$ with $|x|\leq R/10$, let $h_{ij}$ be a Riemannian metric on $B^+_R(x)$, $k>0$ and $\Theta\in L^q(\pa'B_R)$. Denote $Q_2$  as the continuous bilinear form on $H_{0,L}(B^+_R(x))\times H_{0,L}(B^+_R(x))$
\[
 Q_2(\varphi, \phi)=\int_{B^+_R(x)}\nabla_{h}\varphi\nabla_h\phi\,\ud v_h-k\int_{\pa' B^+_R(x)}|\Theta|^{q-3}\Theta \varphi \phi,\ud s_h.
\]
There exists a small positive $\va_0$ depending only $n$ such that if
\[
 \|\Theta-U_{1}\|_{L^q(\pa' B^+_R(x))}+|k-\frac{q-1}{S}|+\|h_{ij}-\delta_{ij}\|_{L^\infty(B^+_R(x))}\leq \va_0,
\]
then
\[
 Q_2(w,w)\geq \frac{c_1}{2}\int_{B^+_R(x)}|\nabla_h w|^2\,\ud v_h,
\]
for all $w\in E_2$, where
\[
 \begin{split}
 E_2=\Big\{w\in H_{0,L}(B^+_R(x)): &\quad \langle U_{0,1}, w\rangle_1\leq \va_0\|w\|_{h},  \langle \frac{\pa U_{0,\lda}}{\pa \lda}\Big|_{\lda=1}, w\rangle_1
\leq \va_0\|w\|_{h} \\&  \langle \frac{\pa U_{\xi,1}}{\pa \xi_i}\Big|_{\xi=0}, w\rangle_1\leq \va_0\|w\|_{h},
, i=1,2,\cdots, n-1 \Big\}.
\end{split}
\]
\end{lem}

\begin{proof}
 From the assumptions, there exist $\delta_1,\cdots,\delta_{n+1}$ satisfying $|\delta_j|=O(\va_0\|w\|_h)$ for every $j$ such that $$\tilde w=w-\sum_{j=1}^{n-1}\delta_{j}\frac{\pa U_{\xi,1}}{\pa \xi_i}\Big|_{\xi=0}
-\delta_n\frac{\pa U_{0,\lda}}{\pa \lda}\Big|_{\lda=1}-\delta_{n+1} U_{0,1}$$ belongs to $E_1$. It follows from Lemma \ref{lem:cowhole} that
\[
 Q_1(\tilde w, \tilde w)\geq  c_1\| \tilde w\|_{\mathcal{D}^{1,2}(\R^n_+) }.
\]
The lemma follows by choosing sufficiently small $\va_0$ and making use of the Sobolev trace inequality.
\end{proof}

\bigskip

\noindent T. Jin

\noindent Department of Mathematics, The University of Chicago\\
5734 S. University Avenue, Chicago, IL 60637, USA\\[1mm]
Email: \textsf{tj@math.uchicago.edu}

\medskip

\noindent J. Xiong

\noindent Beijing International Center for Mathematical Research, Peking University\\
Beijing 100871, China\\[1mm]
Email: \textsf{jxiong@math.pku.edu.cn}

\end{document}